\documentclass[11pt]{amsart}
\usepackage{amssymb,hyperref}
\usepackage{xcolor}
\usepackage{soul}

\usepackage{thmtools} 
\usepackage{cleveref}

\newtheorem{theorem}{Theorem}
\newtheorem{lemma}[theorem]{Lemma}
\newtheorem{proposition}[theorem]{Proposition}
\newtheorem{corollary}[theorem]{Corollary}

\numberwithin{equation}{section}
\theoremstyle{definition}

\newtheorem{example}[theorem]{Example}

\DeclareMathOperator{\Isom}{Isom}
\newcommand{\Z}{\mathbb Z}
\newcommand{\R}{\mathbb R}
\newcommand{\E}{\mathbb E}
\newcommand{\Q}{\mathbb Q}
\newcommand{\RP}{\mathbb {RP}}

\title{On Sections of Quotient maps by Isometries}
\date{\today}

\author{Suyoung Choi}
\address{Department of Mathematics, Ajou University, Suwon 16499, Republic of Korea}
\email{schoi@ajou.ac.kr}

\author{Tao Gong}
\address{Department of Mathematics, University of Western Ontario, 1151 Richmond Street, London, Ontario, N6A 5B7, Canada}
\email{tgong23@uwo.ca}


\begin{document}
\begin{abstract}
Suppose that a discrete group $G$ acts smoothly, properly, and effectively on a connected smooth manifold $M$, and let $q:M\to M/G$ be the quotient map. We prove that if $q$ admits a continuous section, then $G$ is generated by reflections. We also establish converse results for dissecting reflections and give applications to Euclidean and lattice actions.
\end{abstract}

\maketitle

\section{Introduction}\label{section:intro}
Let a group $G$ act on a topological space $M$, and let $q\colon M\to M/G$ be the quotient map.
A subset $D\subset M$ is called a \textbf{strict fundamental domain} if the restriction $q|_D\colon D\to M/G$ is a homeomorphism.
Equivalently, a strict fundamental domain exists if and only if $q$ admits a continuous section.

Such a section exists only for some special cases. For example, 
$$
  \frac{\R^2}{SO(2)}\cong [0,\infty)\times \{0\},\qquad \frac{\R^2}{O(1)\times O(1)}\cong [0,\infty)\times [0,\infty),
$$
where the displayed homeomorphisms identify strict fundamental domains for the corresponding quotient maps. More generally, the second construction extends to all finite linear reflection groups; see \cite{humphreysReflectionGroupsCoxeter1992} for the classical theory of reflection groups. In this paper, we prove that ``existing sections'' and ``reflection groups'' are equivalent under suitable conditions.

Throughout the paper, $M$ is a connected smooth  manifold without boundary, and $G$ is a discrete group that acts smoothly, properly, and effectively on~$M$. 
There exists a complete $G$-invariant smooth Riemannian metric on~$M$ by~\cite[Theorem~0.2]{kankaanrintaProperSmoothGmanifolds2005}.
We fix such a metric throughout the paper.

An element $r\in G$ is called a \textbf{reflection over $M$} if some connected component of~$M^r$ has codimension one.  The main result of this paper is as follows.

\begin{theorem}\label{thm:main}
    If $q\colon M\to M/G$ admits a continuous section, then $G$ is generated by reflections.
\end{theorem}

This paper is structured as follows: In \Cref{section:proof}, we prove \Cref{thm:main}. 
In \Cref{section:applications}, we discuss the converse for dissecting reflections and give applications to Euclidean and lattice actions.

\subsection*{Acknowledgments} 
The first  author was supported by the National Research Foundation of Korea Grant funded by the Korean Government (RS-2025-00521982).
The second author would like to thank Matthias Franz, Nicole Lemire, Vladimir Gorchakov, Krzysztof Pawłowski, and Nikolay Bogachev for discussions. The second author also thanks Jongbaek Song, Michael Davis, and Mikiya Masuda.

\subsection*{Statement on the Use of AI} We used AI tools for proofreading and improving the English, as well as for limited assistance with \Cref{cor:lattice version}. The mathematical ideas and proof strategies are those of the authors.

\section{Proof}\label{section:proof}

For any subset $S$ of $G$, every connected components of the fixed-point set $M^S$ is a totally-geodesic closed submanifold; see \cite[Proposition~5.6.5]{petersenRiemannianGeometry2016}. 
Furthermore, denoting by $M^S_x$ the connected component of $M^S$ containing $x$, we have
\[
   \dim M^S_x = \dim T_x(M^S)=\dim (T_xM)^S=\dim_{\R}\,\bigcap_{g\in S}\left(\ker \left(T_x(g)-\mathrm{id}_{T_xM}\right)\right).
\]
Here are some easy consequences:
\begin{itemize}
    \item the homomorphism $G_x\to O(T_xM)$ assigning $g$ to $T_x(g)$ is injective;
    \item if $\mathrm{codim}\,M^S_x=1$, then  
    \[
    \left\langle S\right\rangle\cong\Z/2:=\left\langle \mathrm{id}_{T_xM^S}\oplus (-1)_{(T_xM^S)^{\bot}}\right\rangle.
    \]
\end{itemize}

Each $x\in M$ has a  $G_x$-neighborhood $U_x$, called a \textbf{slice}, such that $G\times_{G_x}U_x$ forms a $G$-neighborhood of $G(x)$.
We choose each $U_x$ to be the image under $\exp_x$ of a sufficiently small ball centered at $0\in T_xM$.
Here are some easy observations about a slice $U_x$:
\begin{itemize}
    \item any point $y$ in $U_x$ has $G_y\leq G_x$; 
    \item the slice $U_x$ intersects only finitely many fixed-point submanifolds which are connected components of $\{M^g\mid e\ne g\in G\}$.
\end{itemize}
One can refer to \cite[\S~5.6.4]{petersenRiemannianGeometry2016} for details of the slice theory.

Let
$$
    F:=\left\{x\in M \mid G_x=\{e\}\right\} =M\setminus\bigcup_{g\ne e}M^g
$$
be the free locus. 

\begin{lemma}\label{lem:free-locus}
    The free locus $F$ is open and dense.
    Consequently, the restricted quotient map $q_F\colon F\to F/G$ is a principal $G$-bundle.
\end{lemma}
\begin{proof}
    The existence of slices implies that $F$ is open.

   To see that $F$ is dense, we claim that each slice has a free point.
   If a slice contained no free points, choose $\tilde{x}$ in it so that $|G_{\tilde{x}}|$ is minimal.
   Taking a smaller slice $U_{\tilde{x}}$ contained in the original slice, we have $G_y \leq G_{\tilde{x}}$ for every $y\in U_{\tilde{x}}$.
   Then, minimality gives $G_y = G_{\tilde{x}}\ne \{e\}$.
   Hence, $G_{\tilde{x}}$ acted trivially on $U_{\tilde{x}}$. Then the fixed point set $M^{G_{\tilde{x}}}$ would be the whole manifold $M$, contrary to the effective $G$-action.
\end{proof}

\begin{lemma}\label{lem:section-free-components}
    If $q\colon M\to M/G$ admits a continuous section~$s$, then the induced action of $G$ on $\pi_0(F)$ is free.
\end{lemma}
\begin{proof}
    The map $s$ restricts to a section of the principal $G$-bundle $q_F$.
    Then the total space $F$ is equivariantly homeomorphic to $G\times F/G$.  
    It follows that no non-identity element of $G$ preserves a connected component of $F$.
\end{proof}

  The \textbf{reflection subgroup} $W$ of $G$ is defined by
$$
    W :=\langle r\in G \mid r\text{ is a reflection over $M$}\rangle.
$$
The codimension-one components of the sets $M^r$ are called \textbf{reflection hypersurfaces}.

\begin{proposition}\label{prop:reflection-transitivity}
    The reflection subgroup $W$ acts transitively on $\pi_0(F)$.
\end{proposition}
\begin{proof}
    Let $C$ and $C'$ be connected components of $F$, and choose  a smooth path $\gamma$ in $M$ from an interior point of $C$ to one of $C'$. 

     The path $\gamma$ is covered by finitely many slices $\{U_x,\, x\in \alpha\}$. Then we can get an open neighborhood $U$ of $\gamma$ whose closure is also covered by those slices. There are finitely many fixed-point submanifolds $\{N_i,\,i\in\beta\}$ that intersect $U$. Then we can perturb $\gamma$ inside $U$ to get a new path $\tilde{\gamma}$ such that $\tilde{\gamma}$ meets transversely those submanifolds
     $\{N_i,\,i\in\beta\}$. Consequently, the path $\tilde{\gamma}$ only meets those reflection hypersurfaces, at finitely many points, each lying on a unique reflection hypersurface.

    Set $C_0=C$.
    Suppose these crossings occur in order at points $z_1,\ldots,z_k$, let $r_i$ be the unique reflection generating $G_{z_i}$.
    If $C_i$ denotes the component containing the path immediately after the crossing at $z_i$, the local normal form of $r_i$ gives $r_i C_{i-1}=C_i$.
    Then $C'=(r_k\cdots r_1)C$.
    Since each $r_i$ belongs to $W$, the action of $W$ on $\pi_0(F)$ is transitive.
\end{proof}

\begin{proof}[\textbf{Proof of \Cref{thm:main}}]
    By \Cref{lem:free-locus}, the set $F$ is nonempty, so we may fix a connected component $C$ of $F$. 
    Let $g\in G$. Due to \Cref{prop:reflection-transitivity},
   there is some $w\in W$ such that $wC=gC$.
    Hence $w^{-1}g$ preserves $C$. 
    It follows from \Cref{lem:section-free-components} that $w^{-1}g=e$.
    Thus $g=w\in W$, as desired.
\end{proof}

\begin{corollary}
  If $q\colon M\to M/G$ admits a continuous section~$s$, then the strict fundamental domain $D$, as well as $F/G$, is connected.
\end{corollary}
\begin{proof}
    Because $G\times (F/G)\cong F$ and $G$ acts transitively on $\pi_0(F)$.
\end{proof}

The converse direction of \Cref{thm:main} does not hold in general.
\begin{example}\label{ex:rp2}
    Let $M=\RP^2$ and let $G=\langle r\rangle\cong\Z/2$, where $r$ is reflection in a projective line.
    Since the fixed-point set consists of that projective line and one isolated point, $r$ is a reflection, and hence $G=W$.
    
    Here $F$ is connected. 
    If a section existed, \Cref{lem:section-free-components} would imply that the nontrivial group $G$ acts freely on the one-point set $\pi_0(F)$, which is impossible.
\end{example}

\section{Generalizations and Applications}\label{section:applications}

A reflection~$r$ is called \textbf{dissecting} if $M\setminus M^r$ is  not connected. 
In this case, $M^r$ is a closed submanifold of codimension one, and $M\setminus M^r$ has two connected components; see \cite[Theorem~3]{neebSymmetricSpacesDissecting2022}.
\begin{corollary}\label{cor:dissecting}
    Assume additionally that one of the following conditions holds:
    \begin{itemize}
        \item every reflection in $G$ is dissecting;
        \item the manifold $M$ is simply connected.
    \end{itemize}
    Then $q\colon M\to M/G$ admits a continuous section if and only if $G$ is generated by reflections.
\end{corollary}
\begin{proof}
The two conditions both guarantee that all reflections in $G$ are dissecting; see \cite[Theorem~2.8]{alekseevskyReflectionGroupsRiemannian2007}.
The forward implication follows from \Cref{thm:main}.
    Conversely, if $G$ is generated by dissecting reflections, then 
    \cite[Theorem~4.1.(vi)]{davisGroupsGeneratedReflections1983} provides a strict fundamental domain.
\end{proof}

We can also complete those classical results on Coxeter groups; see for example \cite[Theorem~1.12, Theorem~4.8]{humphreysReflectionGroupsCoxeter1992}.

\begin{corollary}\label{cor:euclidean}
    Let $G\leq \Isom(\E^n)$ be a discrete subgroup.
    Then $\E^n\to \E^n/G$ admits a continuous section if and only if $G$ is generated by affine reflections.
\end{corollary}

We can easily generalize to the discrete case of the strict fundamental domain; that is, the restriction $q|_D$ is bijective. 
\begin{corollary}[{cf. \cite[Theorem~2]{farkasStretchedWeightLattices1984}}]\label{cor:lattice version}
    Let $G\leq O(\R^n)$ be a finite subgroup preserving $\Z^n$.
Then $G$ is a Weyl group if and only if the $G$-action on $\Z^n$
has a discrete strict fundamental domain of the form
$D\cap\Z^n$, where $D\subset\R^n$ is a rational polyhedral cone.
\end{corollary}
\begin{proof}
Here it suffices to prove the ``if" direction.
 Suppose that $D\subset \R^n$ is a rational polyhedral cone such that
$D\cap\Z^n$ is a discrete strict fundamental domain for the $G$-action on
$\Z^n$. Then $D\cap\Q^n$ is a strict fundamental
domain for the $G$-action on $\Q^n$.
We claim that $D$ is a strict fundamental domain for the $G$-action on
$\R^n$; then by \Cref{cor:euclidean}, $G$ is
generated by reflections, and hence is a Weyl group.

For $x\in\R^n$, choose a sequence $x_i\in\Q^n$ converging to $x$.
For each $i$, there is $g_i\in G$ such that $g_i x_i\in D$. Since $G$ is
finite, after passing to a subsequence we may assume that $g_i=g$ is
constant. As $D$ is closed, $gx\in D$.

For uniqueness, suppose that $x,gx\in D$ for some $g\in G$. Since
$D\cap g^{-1}D$ is again a rational polyhedral cone, its rational points
are dense. Choose $x_i\in (D\cap g^{-1}D)\cap\Q^n$ converging to $x$.
Then $x_i,gx_i\in D\cap\Q^n$, so the uniqueness over $\Q^n$ implies
$gx_i=x_i$. Passing to the limit gives $gx=x$. 
\end{proof}

\bibliographystyle{amsalpha}
\bibliography{references}

\end{document}